\documentclass{amsart}
\title{Indecomposable Objects of $\underline{Rep}(GL_t)$ in Terms of Exterior Powers of the Tautological Object and its Dual}
\author{Christopher Ryba}
\address{Department of Mathematics, Massachusetts Institute of Technology, Cambridge, MA 02139, USA}
\email{ryba@mit.edu}
\usepackage{hyperref}
\usepackage{fullpage}
\usepackage{amsmath}
\usepackage{amsfonts}
\usepackage{amssymb}
\usepackage{ytableau}
\usepackage{subfig}
\usepackage{graphicx, caption}
\begin{document}
\maketitle
\begin{abstract}
The purpose of this short note is to understand the Grothendieck ring of the Deligne category $\underline{\text{Rep}}(GL_t)$. We give a formula for the class of an indecomposable object in terms of exterior powers of the tautological object and its dual.
\end{abstract}
\newtheorem{theorem}{Theorem}[section]
\newtheorem{lemma}[theorem]{Lemma}
\newtheorem{proposition}[theorem]{Proposition}
\newtheorem{corollary}[theorem]{Corollary}

\section{Introduction}
\noindent
The Deligne category $\underline{\text{Rep}}(GL_t)$ ``interpolates'' representations of general linear groups $GL_n(\mathbb{C})$ in a precise way. It is a tensor category that admits ``specialisation'' functors to representation categories of general linear groups (of all sizes), which may be used to study the representations of $GL_n(\mathbb{C})$ for all $n$ simultaneously. We discuss the structure of the Grothendieck ring of $\underline{\text{Rep}}(GL_t)$. In particular, we explain how the class of an indecomposable object $X_{\lambda, \mu}$ can be expressed in terms of exterior powers of the tautological object $V$ of $\underline{\text{Rep}}(GL_t)$ as well as exterior powers of $V^*$, the dual of $V$. Our main formula (Equation 4) is:
\[
[X_{\lambda, \mu}] = S_{\lambda,\mu}(x,y) = \sum_{\tau}(-1)^{|\tau|}s_{\lambda / \tau}(x) s_{\mu / \tau^\prime}(y),
\]
which realises the class of the indecomposable object $X_{\lambda, \mu}$ as a particular symmetric function in $\Lambda \otimes \Lambda$ ($\Lambda$ is the ring of symmetric functions), which is identified with the Grothendieck ring of $\underline{\text{Rep}}(GL_t)$. In this setting, $e_i(x) = e_i \otimes 1 = [\bigwedge^i V]$, and $e_i(y) = 1 \otimes e_i = [\bigwedge^i V^*]$.
\newline \newline \noindent
In Section 2, we give a brief summary of the Deligne category $\underline{\text{Rep}}(GL_t)$, and some of its key properties that we will use. In Section 3, we discuss the Grothendieck ring of $\underline{\text{Rep}}(GL_t)$, and in particular that it is isomorphic to the tensor product of two copies of the ring of symmetric functions. In Section 4, we use a Jacbi-Trudi determinant to express the class of $X_{\lambda, \mu}$. In Section 5, we evalate the determinant, and finally, in Section 6, we express our results in a generating function.

\section{Background about $\underline{\text{Rep}}(GL_t)$}
\noindent
The category $\underline{\text{Rep}}(GL_t)$ is a symmetric tensor category depending ``polynomially'' on the parameter $t \in \mathbb{C}$ which ``interpolates'' the representation categories of the general linear groups $GL_n(\mathbb{C})$. It was first introduced in \cite{DM}, with further discussion in \cite{D1}. \cite{deligne}, and \cite{etingofRCR2}. We summarise the definition, focusing on the parts that are important for us. We direct the interested reader to the literature for further details.
\newline \newline \noindent
Let $V = \mathbb{C}^n$ be the tautological representation of $GL_n(\mathbb{C})$, and let $V^*$ be its dual. Then the morphism space 
\[
\hom_{GL_n(\mathbb{C})}(V^{\otimes r_1} \otimes (V^*)^{\otimes s_1}, V^{\otimes r_2} \otimes (V^*)^{\otimes s_2})
\]
can be described using walled Brauer diagrams (which we do not define here). Specifically, provided $n$ is large enough, there is a basis of this hom-space given by walled Brauer diagrams of a fixed shape (for $n$ small, it is a spanning set, but not linearly independent). These diagrams encode combinations of the canonical maps $V \to V$, $V^* \to V^*$, $V \otimes V^* \to \mathbb{C}$, and $\mathbb{C} \to V \otimes V^*$ (where $\mathbb{C}$ indicates the trivial representation).
\newline \newline \noindent
The composition of two morphisms indexed by diagrams $D_1, D_2$ is
\[
D_1 \circ D_2 = n^{c(D_1, D_2)} (D_1 \cdot D_2)
\]
where $D_1 \cdot D_2$ is the concatenation of the diagrams $D_1$ and $D_2$ and $c(D_1, D_2)$ is a combinatorial quantity obtained from the concatenation process. The key observation is that the number $n$ can be replaced with any number (even a formal variable), and the composition still makes sense. So the Deligne category $\underline{\text{Rep}}(GL_t)$ is defined as follows. Consider a category $\mathcal{C}$ whose objects are indexed by pairs of non-negative integers $(r,s)$, which are to be thought of as corresponding to $V^{\otimes r} \otimes (V^*)^{\otimes s}$. The set of morphisms between two such objects is the $\mathbb{C}$-span of walled Brauer diagrams of the appropriate shape, as in the $GL_n(\mathbb{C})$ case. The composition of morphisms is also as in the general linear group case, but now the scalars $n^{c(D_1, D_2)}$ in the composition rule are replaced by $t^{c(D_1, D_2)}$. The category $\mathcal{C}$ is a tensor category with the tensor product of objects given by $(r_1, s_1) \otimes (r_2, s_2) = (r_1+r_2, s_1+s_2)$, and tensor product of morphisms given by an appropriate juxtaposition of diagrams. In fact, $\mathcal{C}$ is rigid and symmetric (as a tensor category), but we do not go into the details here. Finally, $\underline{\text{Rep}}(GL_t)$ is the Karoubian envelope (additive envelope followed by idempotent completion) of $\mathcal{C}$, which can be thought of as allowing direct sums and direct summands of objects. Thus $\underline{\text{Rep}}(GL_t)$ inherits all the structure we have discussed above from $\mathcal{C}$.
\newline \newline \noindent
The indecomposable objects $X_{\lambda, \mu}$ of $\underline{\text{Rep}}(GL_t)$ are indexed by pairs of partitions $(\lambda, \mu)$. When $t=n$ is a positive integer, there is a symmetric tensor functor $F_n: \underline{\text{Rep}}(GL_n) \to GL_n(\mathbb{C})-\text{mod}$. If the lengths of $\lambda$ and $\mu$ are $r$ and $s$ respectively, the image of the indecomposable $X_{\lambda, \mu}$ is an irreducible representation with highest weight determined by $\lambda$ and $\mu$:
\[
F_n(X_{\lambda, \mu}) = V_{(\lambda_1, \ldots, \lambda_r, 0, \ldots, 0, -\mu_s, \ldots, -\mu_1)},
\]
provided that $n \geq r+s$, so that the above sequence defines a valid highest weight (otherwise the image under $F_t$ is the zero object). In particular, as $n$ increases, the highest weight defining $F_t(X_{\lambda, \mu})$ has more zeroes inserted into the region separating the positive and negative parts. It can be checked that tensor products in $\underline{\text{Rep}}(GL_t)$ can be computed by using $F_n$ to pass to $GL_n(\mathbb{C})$-mod for $n$ sufficiently large.

\section{The Grothendieck Ring}
\noindent
Let $R$ be the Grothendieck ring of $\underline{Rep}(GL_t)$. Also let $V$ be the tautological object of $\underline{Rep}(GL_t)$, and $V^*$ its dual. The following proposition is well known, but we explain it for completeness.

\begin{proposition}
We have that $R$ is freely generated (as a polynomial ring) by $[\bigwedge^i V]$ and $[\bigwedge^i V^*]$ (square brackets indicate the taking the class in the Grothendieck ring). This yields an identification $R = \Lambda \otimes \Lambda$ where $\Lambda$ is the ring of symmetric functions, $[\bigwedge^i V] = e_i \otimes 1$, and $[\bigwedge^i V^*] = 1 \otimes e_i$.
\end{proposition}
\begin{proof}
We sketch the proof. Put a filtration on $R$ by making $X_{\lambda, \mu}$ have degree $|\lambda|+|\mu|$. Consider the action on $R$ of multiplying by $[\bigwedge^i V]$ or $[\bigwedge^i V^*]$ (where $i \in \mathbb{Z}_{>0}$). By considering the specialisation functors $F_n$ for $n$ large, we can compute these operations inside $GL_n(\mathbb{C})$-mod for large $n$, where we may use the Littlewood-Richardson rule, and more specifically, the Pieri rule (note that the exterior powers of $V$ and $V^*$ are irreducibles corresponding to partitions with a single column). The Pieri rule says how to take the tensor product of an arbitrary irreducible with one whose diagram is a single column with $r$ boxes (one adds $r$ boxes to the diagram of the first irreducible, no two in the same row, so as to get a partition; all such obtained partitions arise with multiplicity one). The observation is that multiplying $[X_{\lambda, \mu}]$ by $[\bigwedge^i V^*]$ gives a sum of $[X_{\nu, \mu}]$ where $\nu$ is obtained from $\lambda$ using the Pieri rule, plus lower order terms (where some boxes were used up to make $\mu$ smaller). An analogous situation arises for $[\bigwedge^i V^*]$. Passing to the associated graded, we see that $\text{gr}(R)$ is the free polynomial algebra in our generators. Hence $R$ is generated by these generators, with no relations. That is to say that it is also freely generated as a polynomial algebra by the same generators. Recall that elementary symmetric functions freely generate the ring of symmetric functions. Matching $e_i \otimes 1$ and $1 \otimes e_i$ with$[\bigwedge^i V]$ and $[\bigwedge^i V^*]$ respectively, we obtain an isomorphism $R = \Lambda \otimes \Lambda$.
\end{proof}
\noindent
The purpose of this note is to explain how to express $[X_{\lambda, \mu}]$ in terms of $[\bigwedge^i V]$ and $[\bigwedge^i V^*]$. This amounts to identifying the symmetric function $S_{\lambda, \mu} \in \Lambda \otimes \Lambda$ which is equal to $[X_{\lambda, \mu}]$ (given by the formula stated in the introduction). We also give a generating function that encodes this information.
\newline \newline \noindent
We may write $R = \Lambda \otimes \Lambda$, which we choose to think of as the ring of symmetric functions in two sets of variables, $x$ and $y$. In this presentation, we write $e_i(x)$ for $[\bigwedge^i V]$ and $e_i(y)$ for $[\bigwedge^i V^*]$. We will show the following identity (of symmetric functions in four sets of variables):
\begin{equation}\label{genfn}
\sum_{\lambda, \mu} s_\lambda(\alpha) s_\mu(\beta) S_{\lambda, \mu}(x, y) = \left(\prod \frac{1}{1-x_i\alpha_j}\right) \left( \prod \frac{1}{1-y_i\beta_j}\right) \left( \prod (1 - \alpha_i \beta_j) \right),
\end{equation}
where $s_\nu$ indicates a Schur function, and $S_{\lambda, \mu}(x, y) = [X_{\lambda, \mu}]$. 
\newline \newline \noindent
Some work on related topics sets the variables $y$ to be the inverses of the variables $x$ (which is well motivated if one considers characters), but this specialisation is undesirable for our purposes.

\section{Jacobi-Trudi Formula}
\noindent
Our first step is to construct a Jacobi-Trudi style formula for $S_{\lambda, \mu}$.
\newline \newline \noindent
We need the following fact about representations of $GL_n(\mathbb{C})$; 
\begin{equation}\label{detshift}
\bigwedge\nolimits^{i}(V^*) = (\bigwedge\nolimits^i V)^* = (\bigwedge\nolimits^{n-i} V) \otimes \det(V)^*
\end{equation}
The latter equality is easy to see by calculating characters. Now consider the irreducible representation of $GL_n(\mathbb{C})$ ($n$ large) indexed by the signature $(\lambda_1, \ldots, \lambda_r, 0, \ldots, 0, - \mu_{s}, \ldots, - \mu_1)$ (so that it is associated to the pair of partitions $(\lambda, \mu)$, where $l(\lambda) = r$ and $l(\mu) = s$). We tensor this representation with $\det(V)^q$, where $q \geq \mu_1$, so that we obtain a polynomial representation with signature $\rho = (\lambda_1 +q, \ldots, \lambda_r+q, q, \ldots, q, q-\mu_s, \ldots, q-\mu_1)$. We now pass to characters; the Schur function $s_\rho$ may now be written using a Jacobi-Trudi determinant. Since we are interested in elementary symmetric functions (standing in for exterior powers), we use the following form of the Jacobi-Trudi identity:
\[
s_\rho = \det\left( e_{\rho_i^\prime -i +j}(x) \right)
\]
Note that this expression involves taking the dual partition of $\rho$. It is given by the following formula (where $\mu^\prime$ is padded with zeroes as appropriate).
\[
\rho^\prime = (n - \mu_{q}^\prime, \ldots, n - \mu_1^\prime, \lambda_1^\prime, \ldots, \lambda_{r}^\prime)
\]
Now we tensor withe $q$ copies of the dual of the determinant (undoing the earlier shift). Equation \ref{detshift} implies that we may do this by replacing each $e_{n - \mu_{m+1-i}^\prime-i+j}(x)$ with $e_{\mu_{m+1-i}^\prime+i-j}(y)$, and not changing the terms associated to the $\lambda_i^\prime$. While the usual Jacobi-Trudi determinant features indices that increase along rows, we note that rows that feature the variable set $y$ now decrease along rows.
\newline \newline \noindent
The resulting determinant has several features; it involves the dual partitions $\lambda^\prime, \mu^\prime$ rather than $\lambda, \mu$, and it involves elementary symmetric functions rather than complete symmetric functions (which usually occur in the most common form of the Jacobi-Trudi identity). To fix these things, we apply the involution $\omega$ (an algebra automorphism on $\Lambda$ which interchanges elementary and complete symmetric functions) to the generating function in Equation \ref{genfn}, separately in all four variable sets. Doing this for the $\alpha$ and $\beta$ results in restoring $\lambda$ and $\mu$ (rather than the dual partitions), whilst doing this for $x$ and $y$ results in exchanging elementary symmetric functions for complete symmetric functions. We will undo this operation later. Hence, let us write
\begin{equation} \label{det}
\tilde{S}_{\lambda, \mu} = (\omega_x \otimes \omega_y) S_{\lambda^\prime, \mu^\prime}(x,y) = \det 
\left( \begin{array}{ccccccc}
h_{\mu_m}(y) & h_{\mu_m -1}(y) & \cdots & & & &\\
h_{\mu_{m-1} + 1}(y) & h_{\mu_{m -1}}(y) & \cdots & & & & \\
\cdots & \cdots & \cdots & & & \\
 &  &  &\cdots  & h_{\lambda_1}(x) & h_{\lambda_1+1}(x)& \cdots \\
 &  &  &\cdots  & h_{\lambda_2-1}(x) & h_{\lambda_2}(x)& \cdots \\
 &  &  &\cdots  & \cdots & \cdots & \cdots 
 \end{array} \right)
\end{equation}
This means that the result of applying the involutions $\omega$ in each variable set to the generating function in Equation 1 gives the following generating function.
\[
\sum_{\lambda, \mu} s_\lambda(\alpha) s_\mu(\beta) \tilde{S}_{\lambda, \mu}(x, y)
\]
Here we use the fact that $\omega(s_\rho) = s_{\rho^\prime}$, but then re-index the sum, removing all the dual partitions from the expressions.

\section{Calculating $\tilde{S}_{\lambda, \mu}(x,y)$}
\noindent
We now use the extended Jacobi-Trudi formula to find an expression for $\tilde{S}_{\lambda, \mu}$. The extended Jacobi-Trudi formula gives an expression for skew-Schur functions. Namely, if $\rho, \sigma$ are partitions, then $s_{\rho / \sigma} = \det(h_{\rho_i -i - \sigma_j +j})$. The case where $\sigma$ is the empty partition gives the usual Jacob-Trudi determinant. We use the Laplace expansion for the determinant in Equation \ref{det}, but we group terms by cosets of $S_m \times S_n$ in $S_{m+n}$ (recall that $l(\mu) = m$ and $l(\lambda) = n$, so that the matrix is of size $m+n$ by $m+n$). The purpose of these groupings is that the sum of terms in each individual grouping (corresponding to a fixed coset) is (up to sign) the product of a determinant in complete symmetric functions of $y$, and a determinant in complete symmetric functions of $x$. The extended Jacobi-Trudi formula allows us to explicitly describe these.
\newline \newline \noindent
As an example, the trivial coset (i.e. $S_m \times S_n$) gives rise to 
\[
\det
\left( \begin{array}{cccc}
h_{\mu_m}(y) & h_{\mu_m -1}(y) & \cdots & h_{\mu_m - m+1}(y)\\
h_{\mu_{m-1} + 1}(y) & h_{\mu_{m -1}}(y) & \cdots & h_{\mu_{m-1}-m+2}(y) \\
\cdots & \cdots & \cdots & \cdots \\
h_{\mu_1 + m-1}(y) & h_{\mu_1 + m - 2}(y) & \cdots & h_{\mu_1}(y)
 \end{array} \right)
\det
\left( \begin{array}{cccc}
h_{\lambda_1}(x) & h_{\lambda_1 +1}(x) & \cdots & h_{\lambda_1 + n-1}(x)\\
h_{\lambda_{2} - 1}(x) & h_{\lambda_2}(x) & \cdots & h_{\lambda_2 + n -2}(x) \\
\cdots & \cdots & \cdots & \cdots \\
h_{\lambda_n - n+1}(x) & h_{\lambda_n - n + 2}(x) & \cdots & h_{\lambda_n}(x)
 \end{array} \right)
  \]
The first determinant, upon conjugating by the matrix with ones on the anti-diagonal, gives the Jacobi-Trudi form for $s_\mu(y)$, whilst the second determinant is already in the Jacobi-Trudi form for $s_\lambda(x)$.
\newline \newline \noindent
The general term associated to a coset involves choosing $m$ columns out of the $m+n$ columns (determined by the coset, e.g. if the coset is $g S_m \times S_n$, then the columns are $g(1), g(2), \cdots g(m)$). One gets the minor associated to those columns and the first $m$ rows multiplied by the minor associated to the remaining columns and the bottom $n$ rows, multiplied by the sign of $g$, which we take to be a minimal-length coset representative.
\newline \newline \noindent
To work out what these terms look like, we associate to each coset a sequence of $m$ copies of the symbol $\times$ and $n$ copies of the symbol $\circ$ by putting $\times$ in position $g(1), g(2), \cdots, g(m)$, and putting $\circ$ in the other locations. An example sequence is $\times \times \circ \times \circ \circ \times \circ$. Now we construct a partition $\tau$ in the following way. The number of parts of size $i$ in $\tau$ is equal to the number of $\circ$ symbols in the diagram having exactly $i$ copies of the symbol $x$ to their right. Thus in the example sequence, we get the partition $\tau = (2,1,1,0)$. We may also form the analogous partition for the $\times$ symbols (counting the number of $\circ$ symbols to the left), but this is just $\tau^\prime$ because the $i$-th part is equal to the number of $\circ$ symbols to the left of the $i$-th $\times$ symbol from the right. But this is also equal to the number of $\circ$ symbols with at least $i$ $\times$ symbols to their right.
\newline \newline \noindent
We also comment that swapping an adjacent $\times$ and $\circ$ changes the sign of $g$ (minimal length coset representative), and also changes $|\tau|$ by $1$. Since in the case $g= Id$, $\tau$ is the trivial partition, we see that the sign of $g$ is equal to $(-1)^{|\tau|}$ by an inductive argument.
\newline \newline \noindent
The whole point of this construction is that the extended Jacobi-Trudi formula shows that the contribution from $g S_m \times S_n$ is precisely $s_{\mu / \tau^\prime}(y) s_{\lambda / \tau}(x) (-1)^{|\tau|}$. We must sum this over all $\tau$ that can be constructed in this way. Observe that this condition is equivalent to $l(\tau) \leq r$ and $\tau_1 \leq s$ (i.e. that $\tau$ should fit in an $r$ by $s$ rectangle). If the first condition fails, then  $ s_{\lambda / \tau}(x) =0$, and if the second condition fails, then  $s_{\mu / \tau^\prime}(y)=0$. We may therefore sum over all partitions $\tau$, because the terms not included in the original sum are all zero. We hence obtain
\begin{equation}
\tilde{S}_{\lambda,\mu}(x,y) = \sum_{\tau}(-1)^{|\tau|}s_{\lambda / \tau}(x) s_{\mu / \tau^\prime}(y).
\end{equation}
At this point, it is clear that applying the involution $\omega$ in variable sets $x$ and $y$ sends $\tilde{S}_{\lambda,\mu}(x,y)$ to $\tilde{S}_{\lambda^\prime, \mu^\prime}(x, y)$, and therefore $\tilde{S}_{\lambda, \mu}(x,y) = S_{\lambda, \mu}(x,y)$. Comparing this to the generating function we are trying to calculate shows that it is actually invariant under simultaneous application of the involution $\omega$ in all four variable sets. 

\section{The Generating Function}
\noindent
We use some generalised Cauchy identities to deduce Equation \ref{genfn}, they can be found in Chapter 1, Section 5, Example 26 of \cite{macdonald}.
\[
\sum_{\lambda}s_\lambda(\alpha) s_{\lambda / \tau}(x) = s_\tau(\alpha) \left(\prod \frac{1}{1 - x_i \alpha_j} \right)
\]
\[
\sum_{\mu} s_{\mu}(\beta) s_{\mu / \tau^\prime}(y) = s_{\tau^\prime}(\beta) \left( \prod \frac{1}{1 - y_i \beta_j} \right)
\]
Multiplying these two things together, and including the factor of $(-1)^{|\tau|}$ and using (a slightly modified form of) the dual-Cauchy identity, namely
\[
\sum_{\tau} (-1)^{|\tau|} s_\tau(\alpha) s_{\tau^\prime}(\beta) = \prod (1 - \alpha_i \beta_j)
\]
allows us to obtain the claimed formula for the generating function:
\[
\sum_{\lambda, \mu} S_{\lambda, \mu}(x,y) s_\lambda (\alpha) s_\mu(\beta)
=
\left(\prod \frac{1}{1 - x_i \alpha_j} \right)\left( \prod \frac{1}{1 - y_i \beta_j} \right)\prod (1 - \alpha_i \beta_j).
\]

\bibliographystyle{alpha}
\bibliography{ref2.bib}

\end{document}